\documentclass[reqno]{amsart}

\usepackage{a4wide,mathrsfs,color,amsrefs} 
\usepackage[colorlinks=false]{hyperref}

\title{Initial data and black holes for matter models}

\author[A.Y.~Burtscher]{Annegret Y.\ Burtscher} 
\address{Department of Mathematics, 
Radboud University, PO Box 9010, Postvak 59, 6500 GL Nijmegen, The Netherlands}
\email{\href{mailto:burtscher@math.ru.nl}{burtscher@math.ru.nl}}

\numberwithin{equation}{section}

\theoremstyle{plain}
\newtheorem{theorem}{Theorem}[section]
\newtheorem{corollary}[theorem]{Corollary}

\newtheorem{lemma}[theorem]{Lemma}

\theoremstyle{definition}
\newtheorem{definition}[theorem]{Definition}

\theoremstyle{remark}
\newtheorem{remark}[theorem]{Remark}

\newcommand{\s}{{(0)}}
\newcommand{\p}{{(1)}}

\def\RR{\mathbb{R}}

\newcommand{\g}{\ensuremath{\mathbf{g}}}

\begin{document}


\begin{abstract}
 
 To observe the dynamic formation of black holes in general relativity, one essentially
 needs to prove that closed trapped surfaces form during evolution from initial data that do not already contain trapped surfaces. We discuss the recent development of the construction of such admissible initial data for matter models. In addition, we extend known results for the Einstein equations coupled to perfect fluids in spherical symmetry and with linear equation of state to unbounded domains. Polytropic equations of state and regularity issues with the direct application of the singularity theorems in general relativity are discussed briefly. 
\end{abstract}

\maketitle


\section{Introduction}
\label{intro}

The Einstein equations in general relativity, with speed of light and Newton's gravitational constant normalized to $1$, read
 \begin{align}\label{Einstein}
  G_{\mu\nu} = 8 \pi T_{\mu\nu},
 \end{align}
where the left hand side, the so-called Einstein tensor, is given in terms of the Ricci curvature and scalar curvature, $G_{\mu\nu} = R_{\mu\nu} - \frac{1}{2} R \, g_{\mu\nu}$, and the right hand side is the energy-momentum tensor of a particular matter model. Solutions to this equation are four-dimensional manifolds $M$ with Lorentzian metric tensors $\g$, describing how light and particles travel in our universe. The first results on the local existence and uniqueness of solutions (for the vacuum equations) have been obtained in 1952 by Choquet--Bruhat \cite{F}. Ever since, the global behavior of solutions is in the focus of attention.

 Singular solutions are known since the discovery of the Schwarzschild solution in 1916, however, only several decades later the systematic study of singularities and black holes has taken off. According to Penrose's Singularity Theorem from 1965 (see, e.g., \cite{HE,SG}), a spacetime $(M,g)$ is null geodesically incomplete if the following three conditions are met:
\begin{enumerate}
 \item $R_{\mu\nu} X^\mu X^\nu \geq 0$ for all null vectors $X^\mu$,
 \item there is a non-compact Cauchy surface in $M$, and
 \item there is a closed trapped surface in $M$.
\end{enumerate}

 The first two conditions are met by any reasonable matter model, as the first condition is tied to the strong energy condition. The third condition, although tangible, is difficult to verify in general circumstances. It is therefore necessary to have some control over the parameters that illustrate trapping throughout the spacetime, initially as well as during evolution. We briefly examine how this was achieved for certain matter models. We will not discuss the vacuum case here, as it differs significantly from the treatment for matter models (spherically symmetry cannot be employed due to Birkhoff's Theorem) and an excellent review has already been written by Bieri~\cite{B}.

 The first time gravitational collapse was observed in the homogeneous spherically symmetric dust model by Oppenheimer and Snyder in 1939. Initially, the dust collapses into a region $r < 2M$, and then the scalar curvature at the singularity in the center blows up. Only later, however, singularities came into the picture and the term \emph{black hole} was coined by Wheeler.
 
 In a series of papers in the 1990s, Christodoulou considered global existence, uniqueness and regularity of solutions to the Einstein equations coupled to a massless scalar field in spherical symmetry. In \cite{C1} he provided conditions on the initial data that guaranteed the formation of trapped surfaces during evolution, and also proved the weak cosmic censorship conjecture in this case \cite{C2}. The large accumulation of mass in a controlled annular region guaranteed the existence of a trapped surface, even if the initial conditions were far from containing a trapped surface (though more time is required in case the initial data are close to flat). The initial conditions are specified on a future null cone and stated in terms of the mass function $m$ and radius $r$.

 In a more realistic setting, Rendall~\cite{R} and more explicitly Andr\'easson, Kunze and Rein~\cite{AKR} considered the gravitational collapse of collision gas modeled by the Vlasov equation in Schwarzschild coordinates. In astrophysics this model is used to describe galaxies and globular clusters. The solutions of the Einstein--Vlasov system are smooth, and no singularities occur for small initial data. Andr\'easson and Rein proved the formation of trapped surfaces in generalized Eddington--Finkelstein coordinates later in \cite{AR}. The benefit of the latter coordinates is that they can be used to cover the whole spacetime and do not break down at the event horizon. With the advanced null coordinate $v$ and area radius $r$, dynamical spherically symmetric spacetimes are of the form
 \begin{align}\label{EFcoord}
  \g = - a(v,r)b^2(v,r) dv^2 + 2 b(v,r) dv dr + r^2 (d\theta^2 + \sin^2 \theta d\varphi^2 ).
 \end{align}
 Asymptotic flatness is tied to the condition
 \begin{align}\label{AF}
  \lim_{r\to\infty} a(v,r) = \lim_{r\to\infty} b(v,r) = 1.
 \end{align}
 A trapped surface $\{v_\bullet\} \times \mathbb{S}^2(r_\bullet)$ is present if $a(v_\bullet,r_\bullet)<0$. Overall similar to the work of Christo\-doulou, the authors constructed suitable initial data leading to the formation of trapped surfaces out of spherically symmetric steady states at the center that are surrounded by a shell of matter moving inwards. The particle density was the key property that was adjusted to achieve this. Weak cosmic censorship holds for these data due to the work of Dafermos and Rendall~\cites{D,DR}.
 
 In the universe, black holes are expected to form when very massive stars collapse. In general relativity, stellar objects are described by a perfect fluid and modelled by the Einstein--Euler equations \eqref{Einstein}, where $T_{\mu\nu}$ is the energy-momentum tensor of a perfect fluid, given in terms of the pressure $p$, density $\rho$ and velocity vector field $u_\mu$, i.e.,
 \[
  T_{\mu\nu} = (\rho+ p) u_\mu u_\nu + p \, g_{\mu\nu}.
 \]
 One of the major difficulties still is to describe the matter-vacuum boundary during evolution \cites{BM,O}. In order to avoid this difficulty at first, LeFloch and the author studied the gravitational collapse of (spherically symmetric) perfect fluids with a priori infinite extent \cite{BLF}. More precisely, the linear equation of state,
 \[
  p = k^2 \rho,
 \]
 for $k \in (0,1)$ representing the (normalized) speed of sound, was employed. In \cite{BLF} it was shown that spherically symmetric steady states can be perturbed in an annular region by manipulating the normalized velocity in a way that while the initial data did not contain trapped surfaces, during the evolution trapped surfaces form. This approach also made use of the generalized Eddington--Finkelstein coordinates \eqref{EFcoord}, however, \eqref{AF} was not (could not) be used due to the unknown asymptotic behavior of static solutions. Thus rather than integrating from spatial infinity the analysis had to be restricted to an (albeit arbitrarily large but nevertheless) compact region. Only in recent work of Andersson and the author~\cite{AB} on spherically symmetric static solutions of the Einstein--Euler equations, it became ultimately clear that perfect fluid solutions with linear equation of state are not asymptotically flat and how \eqref{AF} needed to be modified in order to describe common perfect fluid solutions with infinite extent globally. The situation is different for equations of state that are only piecewise linear, e.g., as studied in the work of Christodoulou \cites{Chr1,Chr2,Chr3} and Fournodavlos and Schlue \cite{FS}, however, no results on 
 the formation of trapped surfaces are known in this setting and we will not discuss it further.
 
 On the following pages, we employ the geometric description derived in \cite{AB} to extend the trapping results for perfect fluids of \cite{BLF} to unbounded domains. We focus here on constructing admissible initial data, since the remaining local existence and trapping analysis based on a generalized random choice scheme and control during evolution can be carried over directly from \cite{BLF}.


\section{Construction of admissible initial data}

 The crucial step in \cite{BLF} is the construction of admissible initial data, that is, initial data that do not contain trapped surfaces but will evolve into solutions that do contain trapped surfaces during their time of existence. The nonexistence of trapped surfaces in the initial data is---in theory---easy to  achieve, since it only requires to check that $a(v_0,r) > 0$ at the initial time $v_0$ for all $r$ in question. In general, $a$ can be computed using the integral representation
\begin{align}\label{aint1}
 a(v,r) = 1 - \frac{4\pi(1+k^2)}{r} \int_0^r \frac{b(v,r')}{b(v,r)} M(v,r') ( 2 k^2 |V(v,r')| +1 ) r'^2 \, dr',
\end{align}
 where $M = b^2 \rho u^0 u^0$ is a normalized mass and $V = \frac{u^1}{bu^0} - \frac{a}{2}$ is a normalized velocity \cite{BLF}*{Sec.\ 2.3}. In practice, however, obtaining this positivity control on $a$ is nontrivial. We investigate this problem in detail.
 
 As mentioned in the Introduction, the idea to obtain admissible initial data is to construct static solutions and then introduce a large but localized perturbation to initiate trapped surface formation. Static solutions satisfy
  \[
  V_\textrm{static} = - \frac{a_\textrm{static}}{2},
 \]
 and do not contain trapped surfaces. The latter property should be preserved, to some extend, even with a large perturbation. Around the center $r=0$ the sign of $a$ is clearly positive due to the integral representation in \eqref{aint1}, however, this property may not hold for large $r$. This problem did not occur in the work of Andr\'easson and Rein~\cite{AR}*{Sec.\ 5}, because due to the asymptotically flat model they used, the ADM mass $M$ was finite and they could simply integrate $a$ from spatial infinity. For \[a(v,r) = 1 - \frac{2m(r)}{r},\] their integral representation \cite{AR}*{Eq.\ (5.2)} from infinity is determined by
 \begin{align}\label{mV}
  m(v,r) = \frac{M}{b(v,r)} - \frac{1}{2} \int_r^\infty 4 \pi \eta^2(T_{11}+S) e^{- \int_\eta^r 4 \pi \rho T_{11} \, d\sigma} \, d\eta,
 \end{align}
 where $S$ depends on the density, the conserved angular momentum and canonical momenta corresponding to the coordinates $(v,r,\theta,\varphi)$.
 
 In the setting of perfect fluids with linear equation of state an analogous integral representation of $a$ is not possible due to the infiniteness of the ADM mass of the static solution. Therefore, in \cite{BLF}, we restricted our attention to solutions on a bounded domain. Recently, Andersson and the author investigated the asymptotic behavior of the static solutions to perfect fluids models with linear and polytropic-type equations of state in more detail. In \cite{AB}*{Thm.\ 1.2} was established that the solutions for linear equations of state are, in fact, asymptotically conical with deficit angle\footnote{Note that in the notation of  \cite{AB} the squared (normalized) speed of sound is denoted by $K=k^2$.} $\alpha = \frac{4k^2}{(1+k^2)^2+4k^2}$ depending solely on the normalized speed of sound $k$. In a spacetime version, this behavior fits into the quasi-asymptotically set-up of Nucamendi and Sudarsky \cite{NS} (see also \cite{BV}), for which an alternative notion of ADM mass has been defined. This so-called ADM$\alpha$ mass is coordinate invariant and thus represents a geometric invariant, however, neither an analogue of the Positive Mass Theorem nor the fact that is constant over time have yet been established. A reasonable premise when dealing with perfect fluids with linear equation of state in general relativity would be to simply \emph{assume} that the solutions are quasi-asymptotically flat. For the kind of initial data we are interested in, this assumption is satisfied due to \cite{AB}*{Thm.\ 1.2} (compact perturbations do not change the asymptotic behavior) and we can replace the use of the integral representation \eqref{mV} in the Vlasov case involving the ADM mass $M$ by employing the deficit angle $\alpha$ in a suitable way.

 \subsection{Asymptotic behavior for static solutions revisited}\label{ssec1}
 
 In order to understand quasi-asymptotic flatness in terms of the metric representation in coordinates $(v,r,\theta,\varphi)$ we rewrite the static solution in these coordinates.
 
 \begin{lemma}[Static solutions in generalized Eddington--Finkelstein coordinates]\label{lem:sasymp}
 Static spherically symmetric solutions of the Einstein--Euler equations for linear equations of state $p=k^2\rho$ are of the form
 \begin{align}\label{gg}
  \g = - a(r) b^2(r) dv^2 + 2 b(r) dv dr + r^2 (d\theta^2 + \sin^2 \theta d\varphi^2)
 \end{align}
 with $a(r) = 1- \frac{2m(r)}{r}$ for the mass function $m$, conical angle $\alpha = \frac{4k^2}{(1+k^2)^2+4k^2}$ and decay
 \begin{align}
  \lim_{r\to\infty} a(r) &= 1-\alpha, \label{alimit} \\
  \lim_{r\to\infty} r^{-\frac{2k^2}{1+k^2}} b(r) &= \left( \frac{2\rho_0}{\pi\alpha} \right)^{\frac{k^2}{1+k^2}} \frac{1}{\sqrt{1-\alpha}}, \label{blimit}
  \end{align}
 \end{lemma}

 \begin{proof}
  According to \cite{AB}*{Cor.\ 2.6} and \cite{AB}*{Cor.\ 3.6} solutions are of the form
  \[
   \g = - e^{2\nu(r)} dt^2 + e^{2\lambda(r)} dr^2 + r^2 (d\theta^2 \sin^2 \theta d\varphi^2)
  \]
  with
  \[
   e^{2\Lambda} := \lim_{r \to \infty} e^{2\lambda(r)} = \lim_{r\to\infty} \left(1 - \frac{2m(r)}{r}\right)^{-1} = \frac{(1+k^2)^2+4k^2}{(1+k^2)^2} = (1 - \alpha)^{-1}
  \]
  and
  $
   \nu'(r) = O(r^{-\frac{1}{2}}) \quad \text{as } r\to\infty.
  $
  We set \[ v := t + \int_0^r e^{\lambda(s) - \nu(s)} \, ds.\]
  Note that the integral converges because, as $r \to 0$ the asymptotic behavior is the metric coefficients is $e^{\lambda(r)} = \left( 1 -\frac{2m(r)}{r} \right)^{-\frac{1}{2}} \sim \frac{1}{\sqrt{1-r^2}} \to 1$ and $e^{\nu(r)} = \left( \frac{\rho_0}{\rho(r)} \right)^{\frac{k^2}{1+k^2}} \sim \left(\frac{\rho_0}{\rho_0}\right)^{\frac{k^2}{1+k^2}} = 1$ (cf.\ \cite{AB}*{(3.3) and Sec.\ 3.1}).
  Thus
  \[
   dv = dt + e^{\lambda(r)-\nu(r)} dr,
  \]
  and therefore
  \[
   e^{2\nu(r)} dt^2 = e^{2\nu(r)} dv^2 - 2e^{\lambda(r)+\nu(r)} dv dr + e^{2\lambda(r)} dr^2.
  \]
  The metric $\g$ in coordinates $(v,r,\theta,\varphi)$ thus is of the form
 \begin{align*}
 \g &= - e^{2\nu(r)} dv^2 + 2 e^{\lambda(r)+\nu(r)} dv dr + r^2 (d\theta^2 \sin^2 \theta d\varphi^2),
 \end{align*}
  which for
  \begin{align}
   b(r) = e^{\lambda(r)+\nu(r)} \quad \text{and} \quad a(r) = e^{-2\lambda(r)} = 1 -\frac{2m(r)}{r}
  \end{align}
  yields the desired form \eqref{gg}.
 By the above and by \cite{AB}*{Cor.\ 3.6} we obtain
 \begin{align*}
  \lim_{r\to\infty} a(r) &= \lim_{r\to\infty} 1 - \frac{2m(r)}{r} = 1 - \alpha, \\
  \lim_{r\to\infty} r^{-\frac{2k^2}{1+k^2}} b(r) &= \lim_{r\to\infty} r^{-\frac{2k^2}{1+k^2}} \left( \frac{\rho_0}{\rho(r)} \right)^{\frac{k^2}{1+k^2}} \left( 1 - \frac{2m(r)}{r} \right)^{-\frac{1}{2}} = \left( \frac{2\rho_0}{\pi\alpha} \right)^{\frac{k^2}{1+k^2}} \frac{1}{\sqrt{1-\alpha}}. \qedhere
 \end{align*}
 \end{proof}
 
 \begin{remark}
  The proof of the asymptotic behavior as $r\to\infty$ is based on the analysis in \cite{AB}. An explicit, so-called singular, solution of the static Einstein--Euler equations in spherical symmetry exists, to which all other solutions are asymptotic as $r \to \infty$. The density of this solution blows up at the center, hence the name ``singular solution''. In \cite{BKTZ} we have shown that this solution is, although singular, still surprisingly well-behaved in a way that it satisfies the second Bianchi identity weakly. The stability of this solution may be studied using metric convergence, e.g., in the sense of Gromov--Hausdorff convergence or Sormani--Wenger intrinsic flat convergence \cites{AlB,B,LFS,SV,SW}.
 \end{remark}

 In a general dynamic setting for the spherically symmetric Einstein--Euler equations with linear equation of state, one can reasonably assume that the initial data have the same asymptotic behavior as that obtained for static solutions in Lemma~\ref{lem:sasymp}. Since we ore only interested in initial data based on static solutions with a compact perturbation, this is not a restriction for our set-up in the next Section.

\subsection{Construction of admissible initial data}

 The idea is to construct admissible initial data for trapped surface formation on an \emph{unbounded} domain. The presentation is inspired by \cite{BLF}*{Sec.\ 6.2}, where an analogous result has been obtained for arbitrarily large but \emph{bounded} domains.

 Let us recall the set-up of \cite{BLF} for constructing admissible initial data for the spherically symmetric Einstein--Euler equations. The main goal was to observe the dynamic formation of trapped surface from untrapped initial data. The property that initial data do \emph{not} contain trapped surfaces requires that
 \begin{align}\label{apositive}
  a(v_0,r) > 0 \qquad \text{for all } r \geq 0
 \end{align}
 initially. In order to observe the formation of trapped surfaces, which corresponds to a sign change, i.e.,
 \[
  a(v_\bullet,r_\bullet) < 0 \qquad \text{for some } v_\bullet, r_\bullet >0,
 \]
 we need to make sure that the initial data, in addition to \eqref{apositive}, also satisfy
 \begin{align*}
  a_v(v_0,r) \ll 0 \qquad \text{for } r \in [r_*-\delta,r_*+\delta] \subseteq [0,\infty),
 \end{align*}
 meaning that the derivative is large and negative in a small region. In \cite{BLF} we proved \eqref{apositive} for arbitrarily large domains $[0,r_*+\Delta]$. The following result, based on on the asymptotic analysis of static solutions in Section~\ref{ssec1}, generalizes it to all of $[0,\infty)$. We start with a definition.
 
 \begin{definition}
  Let $(M^\s,V^\s,a^\s,b^\s)$ be a static solution of the spherically symmetric Einstein--Euler equations with linear equation of state and central density $\rho_0>0$. Let $r_* >0$, $\Delta \in (0,r_*)$ and $\delta \in (0, \Delta)$ and $h>0$ be given. We consider a perturbation of the normalized fluid velocity, defined by a step function
  \[
   V^\p (r) = \begin{cases}
               0 & r < r_* -\delta, \\
               \frac{V^\s(r)}{h} & r_* - \delta \leq r \leq r_* + \delta, \\
               0 & r> r_* +\delta.
              \end{cases}
  \]
  We call $(M_0,V_0,a_0,b_0)$ the \emph{$(r_*,\delta,h)$-perturbed initial data} if 
  \begin{align}
   M_0 = M^\s, \qquad
   V_0 = V^\s + V^\p, \qquad
   b_0 = b^\s,
  \end{align}
  and $a_0$ is given by the integral (cf.\ \cite{BLF}*{Eq.\ (6.9)})
  \begin{align}
  a_0(r) &= 1 - \frac{4 \pi (1+k^2)}{r} \int_0^r \frac{b_0(s)}{b_0(r)} M_0(s) \left( 2\frac{1-k^2}{1+k^2} |V_0(s)| + 1 \right) s^2 \, ds \nonumber \\ 
   &= 1 - \frac{4 \pi (1+k^2)}{r} \int_0^r \frac{b^\s(s)}{b^\s(r)} M^\s(s) \left( 1+ \frac{1-k^2}{1+k^2} \left( 1 +\frac{1}{h} \chi_{[r_*-\delta,r_*+\delta]} \right) a^\s(s) \right) s^2 \, ds. \label{a0}
  \end{align}
 \end{definition}
 
 \begin{theorem}\label{thm}
  Let $(M_0,V_0,a_0,b_0)$ be a $(r_*,\delta,h)$-perturbed initial data set to the spherically symmetric Einstein--Euler equations with linear equation of state $p=k^2\rho$, $k \in (0,1)$ and central density $\rho_0>0$. 
  Then there exist constants $C_1, C_2, C_3, C_4 > 0$ depending on $r_*>0$ and a fixed\footnote{We can also simply choose, for instance, $\Delta = \frac{r_*}{2}$ in order to avoid another parameter.} $\Delta \in (0,r_*)$
  such that for all $\delta,h>0$ with $\frac{\delta}{h} \leq \frac{1}{C_1}$ the following holds:
  \begin{align*}
  & 0 < a_0(r) \leq a^\s(r), \quad r \geq 0, \\
  & \partial_v a_0(r) \begin{cases}
                      = 0 & 0 \leq r < r_* - \delta, \\
                      <0 & r > r_* - \delta, \\
                      \leq - C_2 \frac{\delta}{h^3} & r_* - \delta \leq r  \leq r_* + \delta, \\
                      \leq - C_4 \frac{1}{h^2} & r_* - \delta \leq r  \leq r_* + \delta, \\
                      \leq - C_3 \frac{\delta}{h} & r \in (r_*+\delta, r_* + \Delta].
                     \end{cases}
  \end{align*}
 In particular, this initial data set does not contain trapped surfaces and $\partial_v a_0 \ll 0$ for suitably chosen $\delta$ and $h$.
 \end{theorem}
 
 \begin{proof}
  We proceed as in the proof of \cite{BLF}*{Prop.\ 6.1}. The major difference is Step 1, and we also generalize Step 2 and add an additional Step 5. Steps 3 and 4 can be obtained in the same fashion for a fixed $\Delta \in (0,r_*)$ (or simply $\Delta := \frac{r_*}{2}$).
 
 \smallskip
  {\bf Step 1. Positivity of $a_0$.} Static solutions do not contain trapped surfaces, and thus $a^\s$ is positive throughout. Due to \eqref{a0}, this immediately implies that
  \[
   a_0(r) = a^\s(r) > 0 \qquad \text{for all } r<r_*-\delta.
  \]
  Let $r\geq r_*-\delta$. Then, by \eqref{a0} and for $a^\p := a_0 - a^\s$,
  \begin{align}
   a_0(r) &= a^\s(r) + a^\p(r) \nonumber \\
          &= a^\s(r) - \frac{4\pi(1-k^2)}{rh} \int_{r_*-\delta}^{\min(r,r_*+\delta)} \frac{b^\s(s)}{b^\s(r)} M^\s(s) a^\s(s) s^2 \, ds \nonumber \\
          &\geq a^\s(r) - \frac{4\pi(1-k^2)}{rh} \int_{r_*-\delta}^{r_*+\delta} \frac{b^\s(s)}{b^\s(r)} M^\s(s) a^\s(s) s^2 \, ds, \label{e1}
  \end{align}
  since $M^\s,b^\s,a^\s > 0$. By Lemma~\ref{lem:sasymp}, and the fact that $a$ is monotonically decreasing (cf.\ \cite{BLF}*{Sec.\ 4}) we know that
  \[
   a^\s(r) > 1 - \alpha > 0, \qquad \text{for all } r \geq 0,
  \]
  where $\alpha=\frac{4k^2}{(1+k^2)^2+4k^2}$ is a constant strictly less than $1$ for all $k \in [0,1]$. It thus remains to be shown that the integral term in \eqref{e1} is less than $1-\alpha$. We show that this can be achieved for certain ratios of $\delta$ and $h$. Since, as $r \to \infty$, $b^\s \geq 1$ is increasing and $\rho_0 \geq \rho^\s = a^\s M^\s > 0$ (cf.\ \cite{BLF}*{Eq.\ (4.5)}) is monotonically decreasing by \cite{BLF}*{Thm.\ 4.3}) we obtain that
  \begin{align*}
   0 < - a^\p(r) &\leq \frac{4\pi(1-k^2)}{rh} b^\s(r_*+\delta) a^\s(r_*-\delta) M^\s(r_*-\delta) \left[ \frac{r^3}{3} \right]_{r_*-\delta}^{r_*+\delta} \\
   &\leq \frac{8\pi(1-k^2)}{3} \frac{\delta}{h} \frac{\delta^2+3r_*^2}{r_*-\delta} b^\s(r_*+\delta) \rho_0. 
  \end{align*}
  Without loss of generality we may assume that $\delta \leq \min\{\frac{r_*}{2},\Delta\}$, hence $\frac{\delta^2+3r_*^2}{r_*-\delta} \leq \frac{13r_*}{2}$, so that we obtain
  \[
   0<- a^\p(r) \leq \frac{52 \pi(1-k^2)}{3} \rho_0 \, r_* b^\s\left(\tfrac{3r_*}{2}\right) \frac{\delta}{h}
  \]
  Thus for $\frac{\delta}{h}$ sufficiently small, more precisely, for $\frac{\delta}{h} \leq \frac{1}{C_1}$ with $C_1(r_*,\rho_0,k) := \frac{52 \pi(1-k^2)}{3} \rho_0 \, r_* b^\s\left(\frac{3r_*}{2}\right)(1-\alpha)^{-1}$, we thus obtain that
 \[
   - a^\p(r) \leq 1-\alpha.
 \]
  Therefore, for any $r \geq r_*-\delta$, we have
  \[
   a_0(r) = a^\s(r) + a^\p(r) > 1-\alpha -(1-\alpha) =0.
  \]
  Thus
  \[
   a_0(r) > 0  \qquad \text{for all } r\geq 0,
  \]
  and hence the initial datum does not contain trapped surfaces. 
  
  \smallskip
  {\bf Step 2. Negativity of $\partial_v a_0$.} By \cite{BLF}*{Eq.\ (3.3)} we know that $a$ must satisfy
  \begin{align*}
   a_v(v_0,r) = 2\pi r b^\s(r) M^\s(r) (a_0^2(r) - 4 V_0^2(r)).
  \end{align*}
  By \cite{BLF}*{Thm.\ 4.3}, static solutions satisfy $0 < a^\s = - 2V^\s \leq 1$. Then \eqref{a0} implies\footnote{Note that the calculation \cite{BLF}*{Eq.\ (6.13)} contains two minor typos.}, for any $r \geq 0$,
  \begin{align}\label{e2}
   a_v(v_0,r) &= 2\pi r b^\s(r) M^\s(r) \left((a^\s(r)+a^\p(r))^2 - \left[ a^\s(r) (1+\tfrac{1}{h} \chi_{[r_*-\delta,r_*+\delta]}(r))\right]^2 \right) \nonumber \\
   &= 2\pi r b^\s(r) M^\s(r) \left( a^\p(r) (a_0(r)+a^\s(r)) - \chi_{[r_*-\delta,r_*+\delta]} (a^\s(r))^2 \frac{2h+1}{h^2} \right).
  \end{align}
 Since $a^\s$ and $a_0$ are positive for all $r>0$ by Step 1, and $a^\p$ is negative for $r>r_*-\delta$ by construction, we have that
 \[
  a_v(v_0,r) < 0, \qquad \text{for all } r > r_*-\delta.
 \]

  \smallskip
  {\bf Step 3 and 4. Bounds for $\partial_v a_0$.} One can proceed as in \cite{BLF}*{Prop.\ 6.1} to obtain these bounds.
  
\smallskip
  {\bf Step 5. Additional bound for $\partial_v a_0$ on $[r_*-\delta,r_*+\delta]$.} As in Step 4 of \cite{BLF} one obtains
  \[
   a_v(v_0,r) \leq - 2\pi r b^\s(r) M^\s(r) (a^\s(r))^2 \frac{2h+1}{h^2}.
  \]
  Since $b^\s \geq 1$ is increasing and $\rho^\s = M^\s a^\s$ is decreasing, and $a^\s \geq 1-\alpha$,
  \begin{align*}
   a_v(v_0,r) &\leq - 2\pi (r_*-\delta) \rho^\s(r_*+\delta) (1-\alpha) \frac{2h+1}{h^2} \\
   &\leq - \frac{C_4}{h^2}, \qquad\qquad\qquad\qquad \text{for all } r \in [r_*-\delta,r_*+\delta],
  \end{align*}
  where $C_4$ depends on $r_*$, $\delta$ (or $\Delta$, $r_*$),$k$, and $\rho_0$.
 \end{proof}

  Compared to \cite{BLF}*{Prop.\ 6.1}, the above Theorem~\ref{thm} establishes three additional properties. We have shown that
  \begin{enumerate}
   \item $a_0$ is positive for \emph{all} $r \geq 0$ (and not just up to some $r_*+\Delta$),
   \item $a_v < 0$ for \emph{all} $r > r_*-\delta$ (and not just up to some $r_*+\Delta$),
   \item $a_v \leq - C_4 \frac{1}{h^2}$ for $r \in [r_*-\delta,r_*+\delta]$ holds (in addition to the estimate $a_v \leq - C_3 \frac{\delta}{h^3}$).
  \end{enumerate}
  Property (i), in particular, shows that admissible initial data can be constructed that do \emph{not} contain trapped surfaces on the \emph{unbounded domain} $\RR^3$. All other properties of \cite{BLF}*{Prop.\ 6.1} are preserved, so that the same procedure as in \cite{BLF}*{Sec.\ 6 and 7} establishes the dynamic formation of trapped surfaces. The above Theorem~\ref{thm} thus generalizes \cite{BLF} to unbounded domains. For an exact formulation with all assumptions we refer the reader to \cite{BLF}*{Thm.\ 6.4}.
 
 \begin{corollary}
  The initial value problem for the spherically symmetric Einstein--Euler equations with linear equation of state for a class of $(r_*,\delta,h)$-perturbed initial data sets, prescribed on an unbounded Cauchy surface, leads to solutions with bounded variation with the following properties:
  \begin{enumerate}
   \item The spacetime is a spherically symmetric, future development of the initial data set.
   \item The initial hypersurface does not contain trapped surfaces.
   \item The spacetime does contain trapped surfaces.
  \end{enumerate}
 \end{corollary}
 
 \begin{remark}[Generalization to other equations of state]
While no analysis on the formation of trapped surfaces for perfect fluid models 
have yet been performed for equations of state other than the linear one (even in spherical symmetry), the asymptotic behavior of static solutions w.r.t.\ polytropic-type equations of state, that is, equations of state of the form $p = K \rho^{\frac{n+1}{n}}$ with polytropic index $n>5$, has also been described by Andersson and the author in \cite{AB}. These static solutions also have infinite extend and are also not asymptotically flat. Eventually, of course, one would be interested to study the formation of trapped surfaces for \emph{bounded} fluid balls (models of stars). At the moment, this seems out of reach, as no suitable setting is yet available to study such evolution problems with a fluid--vacuum boundary, but may become available in the future \cite{O}.
 \end{remark}

\section{From trapped surfaces to black holes}

While the Penrose Singularity Theorem discussed in the Introduction would yield a singularity based on the existence of a closed trapped surface, this result requires a metric regularity of $C^2$ (and also a generalization requires at least $C^{1,1}$ \cite{KSV}). In \cite{BLF} solutions of bounded variation have been obtained which do not guarantee this regularity for all available derivatives. As such, the Singularity Theorems known today are not directly applicable. It may be possible to either extend the Singularity Theorems or to improve the regularity along the lines of \cite{RT,Rei} of the solutions obtained in \cite{BLF}. 


\end{document}